\documentclass[oneside,a4paper,11pt,notitlepage]{article}
\usepackage[left=0.8in, right=0.8in, top=1.5in, bottom=1.5in]{geometry}
\usepackage[T1]{fontenc}
\usepackage[utf8]{inputenc} 
\usepackage{tocloft}

\AtBeginDocument{\selectfont}

\usepackage[english]{babel}
\usepackage{float}
\usepackage{amssymb}
\usepackage{lmodern}

\usepackage{amsmath}
\usepackage{amsthm}
\usepackage{bm}
\usepackage{mathtools}
\usepackage{braket}
\usepackage{esint}
\newcommand{\abs}[1]{{\left|#1\right|}}

\usepackage{enumerate}
\usepackage{booktabs}
\usepackage{graphicx}
\usepackage{tikz}
\usetikzlibrary{patterns}
\usepackage{multicol}
\usepackage{caption}
\usepackage[skins,theorems]{tcolorbox}
\newcommand{\Hersch}{\mathcal S}
\newcommand{\Simm}{\mathcal A}
\newcommand{\palle}{\mathcal B}
\newcommand{\herintsim}{ \Hersch \cap  \Simm }
\newcommand{\betaposneg}{(-\infty,0)^2\cup(0, +\infty]^2}

\tcbset{highlight math style={enhanced,
colframe=black,colback=white,arc=0pt,boxrule=1pt}}
\def\XXint#1#2#3{{\setbox0=\hbox{$#1{#2#3}{\int}$}
    \vcenter{\hbox{$#2#3$}}\kern-.5\wd0}}

\usepackage[hyphens]{url}
\usepackage{hyperref}
\usepackage{aliascnt}

\hypersetup{
                colorlinks=true,
                linkcolor={magenta},
                citecolor={cyan},
                breaklinks=true}

\theoremstyle{plain}
\newtheorem{teor}{Theorem}
\numberwithin{teor}{section}
\numberwithin{equation}{section}
\newenvironment{teorema}{\begin{teor}}{\end{teor}}

\theoremstyle{definition}
\newaliascnt{defi}{teor}
\newtheorem{defi}[defi]{Definition}
\aliascntresetthe{defi}
\newenvironment{definizione}{\begin{defi}}{\end{defi}}

\theoremstyle{plain}
\newaliascnt{lemma}{teor}
\newtheorem{lemma}[lemma]{Lemma}
\aliascntresetthe{lemma}
 
\theoremstyle{plain}
\newaliascnt{prop}{teor}
\newtheorem{prop}[prop]{Proposition}
\aliascntresetthe{prop}
 
\theoremstyle{plain}
\newaliascnt{conjecture}{teor}

\aliascntresetthe{conjecture}

\theoremstyle{plain}
\newaliascnt{cor}{teor}
\newtheorem{cor}[cor]{Corollary}
\aliascntresetthe{cor}
\newenvironment{corollario}{\begin{cor}}{\end{cor}}
 
\theoremstyle{definition}
\newaliascnt{ex}{teor}

\aliascntresetthe{ex}
 
\theoremstyle{definition}
\newaliascnt{oss}{teor}
\newtheorem{oss}[oss]{Remark}
\aliascntresetthe{oss}
 
\theoremstyle{plain}

\DeclareMathOperator{\R}{\mathbb{R}}

\makeatletter
\newcommand{\myfootnote}[2]{\begingroup
	\def\@makefnmark{}
	\addtocounter{footnote}{-1}
	\footnote{\textbf{#1} #2}
	\endgroup}
\makeatother

\usepackage{hyperref}
\hypersetup{linktoc=none, bookmarksnumbered, colorlinks=true, linkcolor=magenta, citecolor=cyan}

\title{On the Effectless Cut Method for Laplacian Eigenvalues in any dimensions}
\author{Vincenzo Amato, Nunzia Gavitone, Francesca de Giovanni}
\date{}

\newcommand{\Addresses}{{ 
 \bigskip

  \noindent\textit{E-mail address}, V. ~Amato: \texttt{v.amato@ssmeridionale.it} 

   \medskip 
 
  \noindent\textsc{Mathematical and Physical Sciences for Advanced Materials and Technologies, Scuola Superiore Meridionale, Largo San Marcellino 10, 80138 Napoli, Italy. }

 \medskip
 \noindent\textit{E-mail address}, N.~Gavitone: \texttt{nunzia.gavitone@unina.it}

  \medskip 
 
 \noindent\textit{E-mail address}, F.~de Giovanni: \texttt{francesca.degiovanni@unina.it} 
   \medskip

 \noindent\textsc{Dipartimento di Matematica e Applicazioni ``R. Caccioppoli'', Universit\`a degli studi di Napoli Federico II, Via Cintia, Complesso Universitario Monte S. Angelo, 80126 Napoli, Italy.}

 \par\nopagebreak 

}} 

\begin{document}

\maketitle

   \begin{abstract}
In this paper, we study the optimization of the first Laplacian eigenvalue on axisymmetric doubly connected domains under positive Robin boundary conditions. Under additional geometric constraints, we prove that spherical shells maximize this eigenvalue. Our approach combines known isoperimetric inequalities for mixed Laplacian eigenvalues with a higher-dimensional extension of the \emph{effectless cut} technique introduced by Hersch to study multiply connected membranes of given area fixed along their boundaries.

\vspace{5pt}
\end{abstract}

\begin{center}
\begin{minipage}{10cm}
\small
\tableofcontents
\end{minipage}
\end{center}

\section{Introduction}

Shape optimization problems for Laplacian eigenvalues in multiply connected domains have attracted increasing attention in recent years. In this setting, different boundary conditions may be imposed on the two disjoint components of the boundary and, under suitable geometric constraints, a natural question arises as to whether the spherical shell is an optimal configuration.

The answer strongly depends on the interplay between geometry and boundary conditions, which profoundly influences both qualitative properties and sharp quantitative bounds for the first eigenvalue. Such estimates have important applications in mathematical physics, particularly in the study of eigenvalue problems associated with vibrating membranes.

In this paper, we consider two open and convex sets $\Omega_{in} \subset\subset \Omega_{out} \subset \mathbb{R}^n$, and we define the doubly connected domain
\[
\Omega = \Omega_{out} \setminus \overline{\Omega}_{in}.
\]

For given parameters $(\beta_{in}, \beta_{out}) \in \betaposneg$, we consider the first eigenvalue problem
\begin{equation}
    \label{probprinc}
    \begin{cases}
    -\Delta u = \lambda^{RR}(\Omega)u \qquad & \text{in } \Omega, \\
     \displaystyle \frac{\partial u }{\partial \nu} + \beta_{in} u = 0  \qquad & \text{on } \partial\Omega_{in},\\
    \displaystyle \frac{\partial u }{\partial \nu} + \beta_{out} u = 0  \qquad & \text{on } \partial\Omega_{out}.
\end{cases}
\end{equation}
where $\nu$ denotes the outward unit normal.

The first eigenvalue $\lambda^{RR}(\Omega)$ can also be characterized by the Rayleigh quotient
\begin{equation}
    \label{ray}
    \lambda^{RR}(\Omega) = \min_{\substack{w \in H^{1}(\Omega) \\ w \neq 0}} 
    \frac{\displaystyle \int_\Omega |\nabla w|^2\, dx 
    + \beta_{in} \int_{\partial\Omega_{in}} |w|^2\, d\mathcal{H}^{n-1}
    + \beta_{out} \int_{\partial\Omega_{out}} |w|^2\, d\mathcal{H}^{n-1}}
    {\displaystyle \int_\Omega |w|^2\, dx}.
\end{equation}

Throughout the paper, we adopt the following notation. The first eigenvalue of the Laplacian on a doubly connected domain $\Omega$ is denoted by $\lambda^{XY}(\Omega)$, where $X,Y \in \{R,N,D\}$ indicate the boundary conditions imposed on $\partial \Omega_{in}$ and $\partial \Omega_{out}$, respectively (Robin, Neumann, or Dirichlet). For instance, $\lambda^{DD}(\Omega)$ denotes the first eigenvalue with Dirichlet conditions on both boundary components.

Moreover, any eigenfunction is analytic in $\Omega$ and belongs to $C^{1}(\overline{\Omega})$; see \cite[Proposition~2.1]{asma} and \cite[Proposition~1.6]{lena} for explicit references and further discussion.

\medskip
The aim of this paper is to derive optimal bounds for $\lambda^{RR}(\Omega)$, also including the limiting cases $\beta_{in}, \beta_{out} \in \{0,+\infty\}$, which correspond to Neumann and Dirichlet boundary conditions, respectively.

A key step in our approach is to prove the existence of a set $G$ such that 
\[
\Omega_{\mathrm{in}} \subset\subset G \subset\subset \Omega_{\mathrm{out}},
\]
which allows us to decompose $\Omega$ into two subdomains $G_{in} = G \cap \Omega$ and $G_{out} = \Omega \setminus \overline{G}$ without decreasing the first eigenvalue.

In the planar case, this property, proved by Weinberger in \cite{Weinberger} (see also \cite{BBV}), was used by Hersch \cite{hersch} to establish an upper bound for $\lambda^{DD}(\Omega)$. Indeed, by decomposing $\Omega$ into two subdomains so that
\[
\lambda^{DD}(\Omega) = \lambda^{DN}(G_{in}) = \lambda^{ND}(G_{out}),
\]
the problem reduces to proving isoperimetric inequalities for the mixed eigenvalues $\lambda^{DN}(G_{in})$ and $\lambda^{ND}(G_{out})$ separately.

Extending this strategy to higher dimensions is considerably more delicate, since the classical planar tools do not directly generalize. To the best of our knowledge, this remains an open problem.

To overcome this difficulty, we introduce a suitable class of admissible domains.,
\begin{definizione}
\label{assisimm}
We denote by $\Simm$ the class of doubly connected domains 
$\Omega = \Omega_{out} \setminus \overline{\Omega}_{in}$, where 
$\Omega_{out}, \Omega_{in} \subset \mathbb{R}^n$, $n \ge 2$, are open, bounded, axisymmetric sets with $C^{1,1}$ boundary, sharing the same axis of symmetry.\end{definizione}

We can now state our main result.

\begin{teorema}\label{taglio}
Let $\Omega = \Omega_{out} \setminus \overline{\Omega}_{in} \in \Simm$. Let $u \in C^{1}(\overline{\Omega})$ be a positive eigenfunction associated with the first eigenvalue of problem~\eqref{probprinc}. Then there exists an open set $G$ such that
\[
\Omega_{\mathrm{in}} \subset\subset G \subset\subset \Omega_{\mathrm{out}},
\]
and
\[
\frac{\partial u}{\partial \nu}(x) = 0
\qquad \mathcal{H}^{n-1}\text{-a.e.\ on } \partial_\ast G,
\]
where $\partial_\ast G$ denotes the essential boundary of $G$, i.e., the subset of $\partial G$ on which the outward unit normal is well defined (see Appendix \ref{app} for its precise definition).

\end{teorema}

Moreover, a more general version of Theorem~\ref{taglio} can be established for a suitable class of functions, as shown in Proposition~\ref{taglio2}.

In the spirit of Hersch, we use Theorem~\ref{taglio} to establish upper bounds for $\lambda^{RR}(\Omega)$.

The axisymmetry assumption allows us to extend the effectless cut argument to higher dimensions. We also consider a class of domains for which the sharp isoperimetric inequalities required for the mixed eigenvalues are available; see \cite{DPP,Paoli_Piscitelli_Trani}.

\begin{definizione}
\label{doppioconvessi}
Let $\Hersch$ denote the class of doubly connected domains $\Omega = \Omega_{out} \setminus \overline{\Omega}_{in}$, where $\Omega_{out}, \Omega_{in} \subset \mathbb{R}^n$, $n \ge 2$, are open, bounded, and convex sets satisfying
\begin{equation}
\label{nunziadef}
\frac{|\Omega|}{\omega_n} = \left( \frac{P(\Omega_{out})}{n\omega_n} \right)^{\frac{n}{n-1}} - \left( \frac{W_{n-1}(\Omega_{in})}{\omega_n} \right)^{n}.
\end{equation}
\end{definizione}

Our main application is therefore stated for domains in the intersection $\Hersch \cap \Simm$.

\begin{teorema}\label{thm:main}
Let $\Omega \in \Hersch \cap \Simm$ and $(\beta_{in}, \beta_{out}) \in (0, +\infty]^2$. 
Let $A_{R_1,R_2}$, with $0 < R_1 < R_2$, be the spherical shell such that
\[
|A_{R_1,R_2}| = |\Omega|,\qquad 
P(B_{R_2}) = P(\Omega_{out}),\qquad 
W_{n-1}(B_{R_1}) = W_{n-1}(\Omega_{in}).
\]
Then
\begin{equation}
    \label{teo1.4}
    \lambda^{RR}(\Omega) \le \lambda^{RR}(A_{R_1,R_2}).
\end{equation}
Moreover the equality occours in \eqref{teo1.4} if and only if, up to translations, $\Omega=A_{R_1,R_2},$.
\end{teorema}

Before turning to the proof of Theorem~\ref{thm:main}, in Section~\ref{Sezione4} we show that the class $\Hersch \cap \Simm$ is nonempty.

\begin{prop}
\label{nonempclass}
Let $\palle$ be the class of doubly connected domains $\Omega = B_{out} \setminus \overline{B}_{in}$, where $B_{in}\subset B_{out}$ are two balls of $\mathbb{R}^n$, not necessarily concentric. Then $\palle\subsetneq\herintsim$, i.e. there exist doubly connected domains that are neither
concentric nor translated spherical shells.
\end{prop}

Our result may be viewed as a reverse Faber–Krahn type inequality. The classical Faber–Krahn inequality \cite{faber,krahn} states that the ball minimizes the first Dirichlet eigenvalue under volume constraint. Analogous results hold for positive Robin boundary conditions, while the Szegő–Weinberger inequality characterizes the maximizer of the first nontrivial Neumann eigenvalue.

For multiply connected planar domains, pioneering contributions are due to Hersch and Weinberger \cite{hersch,Weinberger}, who proved that under suitable perimeter constraints the annulus maximizes $\lambda^{DD}$. Mixed boundary conditions were subsequently treated in \cite{herschDN,Weinberger}, and other two-dimensional cases in \cite{anoop}. In higher dimensions, Dirichlet–Robin problems were studied in \cite{gavitonepiscitelli}, while the Neumann–Robin and Robin–Neumann cases were addressed in \cite{DPP,Paoli_Piscitelli_Trani}. Exterior domains were considered in \cite{Krejloto1}, and Kesavan’s result \cite{kesavan} is recovered as a particular case of our theorem.

More recently, stability issues in multiply connected domains have been investigated in several directions. Quantitative stability for the first Robin–Neumann eigenvalue is proved in \cite{CPP}. An isoperimetric inequality for the first Steklov–Dirichlet eigenvalue in convex domains with a spherical hole is obtained in \cite{GPPS}. Torsional rigidity and its stability in multiply connected domains are studied in \cite{AB,B}. Steklov–Robin problems in annular domains are analyzed in \cite{GS2}, and local stability for the first Steklov eigenvalue with a spherical obstacle is established in \cite{PPS}.

\medskip

The paper is organized as follows. Section~\ref{Sezione3} is devoted to the construction of the effectless cut. In Section~\ref{Sezione4} we present the applications and the proof of Theorem~\ref{thm:main}. Finally, in Appendix~\ref{app}, we recall the definition and main properties of quermassintegrals and discuss some tools from shape optimization.

\section{Effectless cut in higher dimension}\label{Sezione3}
Let $\Omega \in \Simm$, that is, $\Omega$ belongs to the class of axisymmetric
doubly connected domains introduced in Definition~\ref{assisimm}. Let
$u \in C^{1}(\overline{\Omega})$ be a positive eigenfunction associated with
the first eigenvalue of problem~\eqref{probprinc}, corresponding to parameters
$\beta_{\mathrm{in}}, \beta_{\mathrm{out}}$ satisfying $(\beta_{\mathrm{in}}, \beta_{\mathrm{out}}) \in \betaposneg.$

We remark that the planar case was established by Weinberger~\cite{Weinberger} and later revisited in~\cite{anoop}.

\begin{oss}\label{hopf}
Thanks to the boundary conditions in~\eqref{probprinc} and the regularity of $u$,
Hopf's lemma yields the following sign properties for the normal derivative:
\begin{itemize}
    \item if $(\beta_{\mathrm{in}}, \beta_{\mathrm{out}}) \in (0,+\infty]^2$, then
    \[
    \frac{\partial u}{\partial \nu}(x) < 0
    \qquad \text{on } \partial \Omega;
    \]
    \item if $(\beta_{\mathrm{in}}, \beta_{\mathrm{out}}) \in (-\infty,0)^2$, then
    \[
    \frac{\partial u}{\partial \nu}(x) > 0
    \qquad \text{on } \partial \Omega.
    \]
\end{itemize}
\end{oss}

\subsection{Critical points}
To prove the main result, the classification of the critical points of $u$ is crucial. To this end, the assumption of axisymmetry of $\Omega$ plays a fundamental role. We will exploit the set of critical points of $u \in C^{1}(\overline \Omega)$. 

\begin{lemma}
\label{codimentions}
Let $\Omega \in \Simm$ and let $u$ be a solution to \eqref{probprinc}, with $(\beta_{\mathrm{in}}, \beta_{\mathrm{out}}) \in \betaposneg$. Then the set of critical points of $u$ consists of one of the following:
\begin{enumerate}
    \item finitely many isolated points on the axis of symmetry;
    \item finitely many isolated analytic varieties of codimension $2$;
    \item a connected analytic variety of codimension $1$, which is the boundary of an open set containing $\Omega_{in}$;
\end{enumerate}
or a combination of the above cases.
In the case of a connected analytic variety of codimension $1$, the function $u$ is constant along such a variety.
\end{lemma}
\begin{proof}
Since $u$ is analytic in $\Omega$, it solves
\[
- \Delta u = \lambda^{RR}(\Omega)\, u \quad \text{in } \Omega,
\]
and, being rotationally symmetric, any isolated critical point must lie on the axis of symmetry.

Any other critical point $x$, by the rotational symmetry of $u$, must belong to an isolated critical set of codimension $2$ or $1$. If the first case does not occur, then in the plane through $x$ and the axis of symmetry one can find a sequence of critical points accumulating at $x$, which necessarily belong to a codimension--$1$ variety, denoted by $C_1$.

Since $u$ has no critical points on $\partial \Omega$, the critical set $C$ is contained in $\Omega$. By the regularity of $u$, $C_1$ is the boundary of an open set containing $\Omega_{\mathrm{in}}$.

By connectedness and the Morse--Sard theorem, we conclude that $u$ is constant on $C$; in particular, $C_1$ is a $C^\infty$ manifold.

\end{proof}

Denote by $K$ the set of critical points of $u$, Lemma~\ref{codimentions} can then be stated as
\[
K = C_1 \cup \left( \bigcup_{i=1}^k C_{2,i} \right) \cup \left( \bigcup_{j=1}^h C_{0,j} \right),
\]
where $C_1$ is a variety of codimension $1$, the sets $C_{2,i}$ are isolated varieties of codimension $2$, and the sets $C_{0,j}$ are isolated points located on the axis of symmetry.
Moreover, the lemma ensures that $C_1$ is of class $C^\infty$.

\subsection{Proof of the main results}
The proof of Theorem \ref{taglio} is not immediate. In order to enhance clarity, we proceed by decomposing the argument into a sequence of intermediate lemmas, each addressing a specific component of the overall result.

To construct the effectless cut, for each $x \in \Omega$ we consider the \emph{flow line} of $u$, defined as the solution of
\begin{equation}
    \label{flow}
    \begin{cases}
        \dot z_x(t) = - \nabla u(z_x(t)), & t \in I, \\
        z_x(0) = x,
    \end{cases}
\end{equation}
where $I$ is the maximal interval such that the solution remains in $\overline{\Omega}$.
Due to the regularity of $u$, the flow $z_x(t)$ is unique and analytic in $\Omega$ for every starting point $x \in \Omega$.

We now introduce the sets
\begin{equation}
\label{gin}
    G_{in} = \{ x \in \Omega \mid \exists\, t_x \text{ such that } z_x(t_x) \in \partial \Omega_{in} \},
\end{equation}
and
\begin{equation}
\label{gout}
    G_{out} = \{ x \in \Omega \mid \exists\, t_x \text{ such that } z_x(t_x) \in \partial \Omega_{out} \}.
\end{equation}

As a first step, we establish the following properties.

\begin{lemma}\label{G_aperti}
Let $G_{in}$ and $G_{out}$ be defined as in \eqref{gin} and \eqref{gout}. Then they are open, disjoint, and connected sets.
Moreover,
\[
\partial \Omega_{in} \subset G_{in}
\qquad \text{and} \qquad
\partial \Omega_{out} \subset G_{out}.
\]
\end{lemma}

\begin{proof}
We prove the statement for $G_{in}$; the proof for $G_{out}$ follows by analogous arguments.

Since $u \in C^{1}(\overline \Omega)$ and $\partial u / \partial \nu \neq 0$ on $\partial \Omega_{in}$, there exists $\overline{\delta} > 0$ such that
\[
(\partial \Omega_{in})_{\overline{\delta}} \subset G_{in}.
\]

Let $x \in G_{in}$. Then there exists $t_x$ such that
\[
z_x(t_x) = \bar{x} \in \partial \Omega_{in}.
\]
Define
\[
\gamma := \{ z_x(t) \mid t \in [0,t_x] \},
\qquad
m := \inf_{z \in \gamma} |\nabla u(z)| > 0.
\]

Let $y \in B_x(a\delta)$ for some $0 < a < 1$, to be fixed later, and choose
\[
\delta < \min \{ m, \overline{\delta}/8 \}.
\]
Our goal is to show that $y \in G_{in}$, namely that there exists $\tilde{t}$ such that
\[
z_y(\tilde {t}) \in (\partial \Omega_{in})_{\overline{\delta}}.
\]

Let $\bar{t}$ be such that
\[
z_x(\bar{t}) \in (\partial \Omega_{in})_{\overline{\delta}/2}
\setminus (\partial \Omega_{in})_{\overline{\delta}/4}.
\]
Since $u$ is analytic on compact subsets of $\Omega$, there exists $L > 0$ such that
\[
|\nabla u(z) - \nabla u(y)| \le L |z-y|
\qquad \forall z,y \in \Omega_{\overline{\delta}/8}= \left\{x \in \Omega : \, d (x, \partial \Omega )> \frac{\overline \delta
}{8}\right\}.
\]

We compute
\begin{align*}
|z_x(t) - z_y(t)|
&= \left| x - y - \int_0^t \bigl( \nabla u(z_x(s)) - \nabla u(z_y(s)) \bigr)\, ds \right| \\
&\le |x-y| + \int_0^t |\nabla u(z_x(s)) - \nabla u(z_y(s))|\, ds \\
&\le a\delta + L \int_0^t |z_x(s) - z_y(s)|\, ds.
\end{align*}
By Gronwall's lemma, it follows that
\[
|z_x(t) - z_y(t)| \le a \delta e^{Lt}
\qquad \forall t \in [0,\bar{t}].
\]

Choosing $a = e^{-L\bar{t}}$, we obtain
\[
|z_x(\bar{t}) - z_y(\bar{t})| < \delta,
\]
and therefore
\[
z_y(\bar{t}) \in (\partial \Omega_{in})_{\overline{\delta}},
\]
which implies $y \in G_{in}$.

Hence, we have proved that $G_{in}$ contains $\partial \Omega_{in}$ and is open. Since every point $x \in G_{in}$ can be connected to $\partial \Omega_{in}$ (which is connected), it follows that $G_{in}$ is also connected. 

Moreover, we emphasize that no flow line can connect both $\partial \Omega_{in}$ and $\partial \Omega_{out}$; otherwise, one would obtain a path from $\partial \Omega_{in}$ to $\partial \Omega_{out}$ along which the gradient does not change sign, contradicting Remark~\ref{hopf}. This ensures that $G_{in}$ and $G_{out}$ are disjoint.
\end{proof}

\begin{corollario}
Let $G_{in}$ and $G_{out}$ be defined as in \eqref{gin} and \eqref{gout}. Then
\[
\overline{G}_{in} \cap G_{out}
=
\overline{G}_{out} \cap G_{in}
=
\emptyset.
\]
\end{corollario}

\begin{proof}
Assume by contradiction that $\overline{G}_{in} \cap G_{out} \neq \emptyset$.
Then there exists $z \in \partial G_{in} \cap G_{out}$.
Since $G_{out}$ is open, this implies the existence of a point
\[
\tilde z \in G_{in} \cap G_{out},
\]
contradicting Lemma~\ref{G_aperti}.
\end{proof}

\begin{lemma}\label{G_aperti2}
Let $x_0 \in \partial G_{in} \cap \Omega$. Then
\[
z_{x_0}(t) \in \partial G_{in} \cap \Omega
\qquad \forall t \in \mathbb{R} \cup \{\pm \infty\}.
\]
\end{lemma}

\begin{proof}
Since $G_{in}$ and $G_{out}$ are open and disjoint, if
$x_0 \in \partial G_{in} \cap \Omega$, then
$x_0 \notin G_{in} \cup G_{out}$.
By continuous dependence on the initial data, for every $\varepsilon > 0$,
\[
B_{x_0}(\varepsilon) \cap G_{in} \neq \emptyset
\quad \Longrightarrow \quad B_{z_{x_0}(t)}(\varepsilon) 
 \cap G_{in} \neq \emptyset
\quad \forall t \in \mathbb{R},
\]
which proves the claim.
\end{proof}

\begin{lemma}\label{ginrot}
Let $\Omega \in \Simm$. Then the sets $G_{in}$ and $G_{out}$ defined in
\eqref{gin} and \eqref{gout} belong to $\Simm$.
\end{lemma}

\begin{proof}
Let us assume that the axis of symmetry is the $x_n$-axis. Moreover, let $x = (x',x_n) \in G_{in}$ and let $x''$ be such that $\|x''\| = \|x'\|$. We set $\bar{x} = (x'',x_n)$.  

There exists a rotation $T$ around the $x_n$-axis such that $T(x)=\bar{x}$.
Since $u$ is rotationally symmetric,
\[
T(\nabla u(y)) = \nabla u(T(y)) \qquad \forall y \in \Omega.
\]
If $z_x(t)$ reaches $\partial \Omega_{in}$ at time $t_x$, then
\[
z_{\bar{x}}(t) = T(z_x(t)) \qquad \forall t \in [0,t_x],
\]
and thus
\[
z_{\bar{x}}(t_x) \in \partial \Omega_{in},
\]
which shows that $\bar{x} \in G_{in}$.
The same argument applies to $G_{out}$.
\end{proof}

To avoid internal cracks and construct a set $G$ such that
\[
\Omega_{in} \subset\subset G \subset\subset \Omega_{out},
\]
we define
\[
G^\ast_{in} := \mathrm{Int}(\overline{G}_{in}),
\qquad
G^\ast_{out} := \mathrm{Int}(\overline{G}_{out}).
\]

\begin{lemma}\label{pgin=pgout}
Let us consider the sets $G^\ast_{in}$ and $G^\ast_{out}$, then 
\[
\partial G^\ast_{in} \cap \Omega = \partial G^\ast_{out} \cap \Omega.
\]
\end{lemma}

\begin{proof}
Assume by contradiction that the statement is false.
Then the set
\[
\Omega \setminus (\overline{G}^\ast_{in} \cup \overline{G}^\ast_{out})
\]
is nonempty.
Let $\Omega_0$ be one of its connected components.
Since $G_{in}$ and $G_{out}$ have finite perimeter, so does $\Omega_0$.
Integrating
\[
- \Delta u = \lambda^{RR}(\Omega) u
\]
over $\Omega_0$ and applying the divergence theorem, we obtain
\[
\lambda^{RR}(\Omega) \int_{\Omega_0} u \, dx
=
\int_{\partial_\ast \Omega_0} \frac{\partial u}{\partial \nu}
\, d\mathcal{H}^{n-1}
= 0,
\]
which contradicts the positivity of $u$ and the fact that $\lambda^{RR}(\Omega) \neq 0$.
\end{proof}

We can now prove the existence of the {\it effectless cut} $G$.
\begin{proof}[Proof of Theorem \ref{taglio}]

We define
\[
G := G^\ast_{in} \cup \overline{\Omega}_{in},
\]
and, by Lemma~\ref{pgin=pgout}, we have
\[
G = \Omega_{out} \setminus \overline{G}^\ast_{out} \quad \text{ and } \quad \Omega_{in} \subset \subset G \subset \subset \Omega_{out}.
\]

Since $G$ is a set of finite perimeter, the outer unit normal $\nu_G$ exists at $\mathcal{H}^{n-1}$-almost every point of $\partial_\ast G$. Moreover, by Lemma \ref{G_aperti2}, for every $x \in \partial_\ast G$ we have either $\nabla u(x)=0$ or $\nabla u(x)$ is tangent to $\partial_\ast G$. In particular, in both cases the gradient is orthogonal to the normal, and therefore
\[
\frac{\partial u}{\partial \nu_G}(x)
= \nabla u(x)\cdot \nu_G(x)
= 0
\qquad \text{for } \mathcal{H}^{n-1}\text{-a.e. } x \in \partial_\ast G.
\]
\end{proof}
In the same spirit as \cite{Weinberger}, we observe that Theorem~\ref{taglio} can in fact be stated for more general functions $u$ satisfying suitable properties.
\begin{prop}\label{taglio2}

Let $\Omega = \Omega_{out} \setminus \overline{\Omega}_{in} \in \Simm$. Let $u \in C^{1}(\overline{\Omega})$ be a  function satisfying
\begin{enumerate}
    \item[a.] $u>0$ in $\Omega$ and $\frac{\partial u}{\partial \nu}$ has constant sign on $\partial \Omega$;
    \item[b.] $u$ in analytic in $\Omega$; 
    \item[c.]  $\Delta u$ has constant sign in $\Omega$;
    \item[d.] $u$ is a axisymmetric function with respect to the axis of symmetry of $\Omega$. 
\end{enumerate}
 Then there exists an open set $G$ such that
\[
\Omega_{\mathrm{in}} \subset\subset G \subset\subset \Omega_{\mathrm{out}},
\]
and
\[
\frac{\partial u}{\partial \nu}(x) = 0
\qquad \mathcal{H}^{n-1}\text{-a.e.\ on } \partial_\ast G.
\]
\end{prop}

\begin{proof}
The proof follows the same lines as that of Theorem~\ref{taglio}, with the main differences arising at two points where we crucially use the fact that, in Theorem~\ref{taglio}, $u$ is a solution to a PDE. In particular:
\begin{itemize}
    \item in the proof of Lemma~\ref{codimentions}, where assumptions $b$, $c$, and $d$ are used;
    \item in the proof of Lemma~\ref{pgin=pgout}, where it is enough to apply the divergence theorem together with property $c$.
\end{itemize}
\end{proof}\section{Applications}\label{Sezione4}
Before we move into the proof of Theorem \ref{thm:main}, we show that 
the class $\Hersch\cap\Simm$ is nonempty.

\begin{proof}[Proof of Proposition \ref{nonempclass}]
Let us consider two convex domains
$\Omega_{\mathrm{in}}$ and $\Omega_{\mathrm{out}}$, together with two balls
$B_{\mathrm{in}}$ and $B_{\mathrm{out}}$, all sharing the same axis of symmetry,
such that
\[
B_{\mathrm{in}} \subset\subset \Omega_{\mathrm{in}}
\subset\subset \Omega_{\mathrm{out}} \subset\subset B_{\mathrm{out}}.
\]

We introduce the one-parameter families of convex sets defined by
\[
\Gamma_{\mathrm{out}}(a) = a B_{\mathrm{out}} + (1-a)\,\Omega_{\mathrm{out}},
\qquad
\Gamma_{\mathrm{in}}(b) = b B_{\mathrm{in}} + (1-b)\,\Omega_{\mathrm{in}},
\]
for $a,b \in [0,1]$, where the sums are understood in the sense of Minkowski
addition. For every $a,b \in [0,1]$, the sets $\Gamma_{\mathrm{out}}(a)$ and
$\Gamma_{\mathrm{in}}(b)$ are convex and axisymmetric, with the same axis of
symmetry.

We now evaluate condition~\eqref{nunziadef} for the doubly connected domain
$\Gamma_{\mathrm{out}}(a) \setminus \overline{\Gamma_{\mathrm{in}}(b)}$ by
introducing the function
$$f(a,b) = \frac{|\Gamma_{out}(a)|}{\omega_n}- \frac{|\Gamma_{in}(b)|}{\omega_n} - \left( \frac{P(\Gamma_{out}(a))}{n\omega_n} \right)^{\frac{n}{n-1}} + \left( \frac{W_{n-1}(\Gamma_{in}(b))}{\omega_n} \right)^{n}.$$

By the Aleksandrov--Fenchel inequality, we have
\[
f(1,0) > 0 \qquad \text{and} \qquad f(0,1) < 0.
\]
Since $f$ is continuous on the compact set $[0,1]^2$, there exist
$\bar a, \bar b \in (0,1)$ such that
\[
f(\bar a, \bar b) = 0.
\]

Consequently, the domain
$\Gamma_{\mathrm{out}}(\bar a) \setminus \overline{\Gamma_{\mathrm{in}}(\bar b)}$
belongs to the class $\herintsim$ and it is neither a
concentric nor a translated spherical shell.
\end{proof}

We can now proceed with the proof of Theorem \ref{thm:main}

\begin{proof}[Proof of Theorem \ref{thm:main}]
    Let $u$ be the solution to \eqref{probprinc}. Since $\Omega \in \Simm$, by Section \ref{Sezione3}, there exists a bounded, open, and  connected set $G \supset \Omega_{in}$ with finite perimeter, such that 
    $$
    \frac{\partial u}{\partial \nu}(x)= 0 \qquad\qquad \mathcal{H}^{n-1}-a.e. \text{ in } \partial_\ast G.$$

Since $u$ solves \[
- \Delta u = \lambda^{RR}(\Omega)\, u \quad \text{in } \Omega,
\]by multipling the equation for any $\varphi \in W^{1,2}(G \cap \Omega)$ and integrating over $G \cap \Omega$, we obtain  
\begin{equation}
\label{gintom}
    \int _{G \cap \Omega} \nabla u \nabla \varphi \, dx+\beta_{in} \int_{\partial\Omega_{in}} u \varphi \, d \mathcal{H}^{n-1} = \lambda^{RR}(\Omega)\int_{G \cap\Omega} u \varphi \, dx, \qquad \forall \varphi \in W^{1,2}(G \cap\Omega).
    \end{equation}
Similarly, we get 
\begin{equation}
\label{gmenom}
    \int _{\Omega \setminus \overline G} \nabla u \nabla \varphi \, dx+\beta_{out} \int_{\partial\Omega_{out}} u \varphi \, d \mathcal{H}^{n-1} = \lambda^{RR}(\Omega)\int_{\Omega \setminus \overline G} u \varphi \, dx, \qquad \forall \varphi \in W^{1,2}(\Omega \setminus \overline G).
    \end{equation} 
Lemma \ref{autofpos} and Remark \ref{tutticasi} ensure us that 
$$
\lambda^{RN}(G\cap \Omega)=\lambda^{NR}(\Omega\setminus \overline{G})=\lambda^{RR}(\Omega).$$

Let us study the two obtained mixed eigenvalues involving the Neumann boundary condition on $\partial G$:
\begin{itemize}
    \item by Theorem \ref{DPP}, one has
    \begin{equation}
    \label{dpp}
        \lambda^{RN}(G\cap \Omega)\leq \lambda^{RN}(A_{R_1,R_3});
    \end{equation}

    \item by Theorem \ref{PPT}, one has
    \begin{equation}
        \label{ppt}
        \lambda^{NR}(\Omega\setminus \overline{G})\leq \lambda^{NR}(A_{R_4,R_2}).
    \end{equation}
\end{itemize}

The outer radius of the first annulus and the inner radius of the second one are the same, in fact, by construction the sum of the volumes of the two annuli must be the one of $\Omega$, hence
$$
\displaystyle \frac{|\Omega|}{\omega_n}=R_2^n - R_4^n +R_3^n-R_1^n,
$$
but from \eqref{nunziadef}
$$ \displaystyle \frac{|\Omega|}{\omega_n}= \displaystyle \left(\frac{P(B_{R_2})}{n\omega_ n}\right)^\frac{n}{n-1}-\displaystyle \left(\frac{W_{n-1}(B_{R_{1}})}{\omega_ n}\right)^n= R_2^n -R_1^n.$$
Hence, subtracting the previous two equations, we obtain
$$
R_3=R_4=r.
$$

We now have 
$$
\lambda^{{RR}}(\Omega) \leq \min \{\lambda^{RN}(A_{R_1,r}), \lambda^{NR}(A_{r,R_2})\},
$$
hence

\begin{equation}
    \label{maxmin}
    \lambda^{{RR}}(\Omega) \leq\max_{r \in [R_1,R_2]} \min \{\lambda^{RN}(A_{R_1,r}), \lambda^{NR}(A_{r,R_2})\}.
\end{equation}
Since, by Lemma~\ref{incrdecr}, the maximum involves the minimum of an increasing quantity and a decreasing quantity in $r$, it is achieved when
\[
\lambda^{RN}(A_{R_1,r}) = \lambda^{NR}(A_{r,R_2}).
\]

Since the two quantities are equal, we can "glue" the two eigenfunctions in $A_{R_1,r}$ and $A_{r, R_2}$, to obtain a positive eigenfunction for $\lambda^{RR}(A_{R_1,R_2})$, hence its first eigenfunction.

To conclude the proof, we analyze the equality case in \eqref{teo1.4} and establish rigidity. First, we observe that equality in \eqref{teo1.4} implies equality in \eqref{maxmin}, and consequently in \eqref{dpp} and \eqref{ppt}.

In particular, we obtain
\begin{equation}
\label{==dpp}
\lambda^{RR}(\Omega)=\lambda^{RN}(G\cap \Omega)=\lambda^{RN}(A_{R_1,r})=\lambda^{RR}(A_{R_1,R_2}),
\end{equation}
and
\begin{equation}
\label{==ppt}
\lambda^{RR}(\Omega)=\lambda^{NR}(\Omega\setminus \overline{G})=\lambda^{NR}(A_{r,R_2})=\lambda^{RR}(A_{R_1,R_2}),
\end{equation}

Hence, by Theorem \ref{DPP} we have 
$$G \cap \Omega = A_{R_1,r},$$
while 
Theorem \ref{PPT} gives
$$ \Omega \setminus\overline G = A_{r,R_2}.$$
Since $\Omega = (G \cap \Omega) \cup (\Omega \setminus\overline G)$, the thesis follows.

\end{proof}

\begin{oss}
While the axisymmetry assumption might not be essential and could potentially be removed, this does not seem to be the case for the convexity assumption.

In the planar case, the analogue of Theorem~\ref{thm:main} is stated for simply connected domains $\Omega_{\mathrm{out}}$ and $\Omega_{\mathrm{in}}$, and the simply connectedness assumption cannot be removed. Indeed, homogenization techniques, or the addition of many small balls outside the domain, can easily provide counterexamples to Theorem~\ref{thm:main} in dimension $2$ if the simply connectedness assumption is dropped, as the eigenvalue $\lambda^{RR}(\Omega)$ can be made arbitrarily large (see for instance \cite{ciormur}).

If one attempts to repeat this construction in higher dimensions, it is possible to connect the added balls to the main domain without substantially affecting the eigenvalue. As a consequence, one can construct simply connected domains $\Omega_{\mathrm{out}}$ and $\Omega_{\mathrm{in}}$ for which Theorem~\ref{thm:main} fails.

For this reason, the convexity assumption appears to be a natural and reasonable requirement.
\end{oss}







\appendix 
\section{Appendix}\label{app}

This appendix is devoted to introducing and recalling the notions needed for the 
applications to the isoperimetric inequalities for the class of doubly connected 
domains considered in Theorem~\ref{thm:main}. 

To this end, we first recall the definition of the essential boundary, the quermassintegrals, then analyze 
the radial case, and finally recall the known results for mixed boundary conditions, 
namely Robin--Neumann and Neumann--Robin conditions.
\subsection{Boundary and essential boundary}

We recall the following definition of perimeter.

\begin{definizione}
    Let $E \subset \R^n$. The \textbf{perimeter} of $E$, denoted by $P(E)$, is defined as
    \[
        P(E) = \sup \Biggl\{ \int_E \mathrm{div}\, \varphi \, dx : \varphi \in C_c^\infty(\R^n;\R^n), \ \|\varphi\|_\infty \le 1 \Biggr\}.
    \]
    We say that $E$ has \emph{finite perimeter} if $P(E) < +\infty$.
\end{definizione}

In general, we have $P(E) \ne \mathcal{H}^{n-1}(\partial E)$. We now recall the notion of essential boundary, which allows us to compute the perimeter in a more practical way.

\begin{definizione}
    Let $E \subset \R^n$ be a measurable set. A point $x \in \R^n$ belongs to the \textbf{essential boundary} of $E$, denoted by $\partial_\ast E$, if
    \[
        \lim_{r \to 0^+} \frac{\mathcal{L}^n(B_r(x) \cap E)}{\mathcal{L}^n(B_r(x))} > 0
    \]
    and
    \[
        \lim_{r \to 0^+} \frac{\mathcal{L}^n(B_r(x) \setminus E)}{\mathcal{L}^n(B_r(x))} > 0.
    \]
\end{definizione}

It is well known (see, e.g., \cite[Theorem 5.15]{evansgar}) that
\[
    P(E) = \mathcal{H}^{n-1}(\partial_\ast E).
\]

For sets of finite perimeter, the divergence theorem still holds, see \cite[Theorem 5.16]{evansgar}.

\begin{teorema}[\cite{evansgar}, Theorem 5.16]
    Let $E \subset \R^n$ be a set of finite perimeter. Then for $\mathcal{H}^{n-1}$-a.e.\ $x \in \partial_\ast E$, there exists a unique measure-theoretic unit outer normal $\nu_E(x)$ such that
    \[
        \int_E \mathrm{div}\, \phi \, dx = \int_{\partial_\ast E} \phi \cdot \nu_E \, d\mathcal{H}^{n-1}
    \]
    for all $\phi \in C_c^1(\R^n;\R^n)$.
\end{teorema}

\subsection{Quermassintegrals}
We briefly summarize the main properties of quermassintegrals for convex sets,
following the standard reference \cite{Schneider_2013}. Let $K \subset \mathbb{R}^n$
be a nonempty, bounded, convex set, and let $B$ be the unit ball centered at the
origin. For $\rho>0$, the volume of the Minkowski sum $K+\rho B$ admits the
expansion
\begin{equation}
\label{Steiner_formula}
|K + \rho B| = \sum_{i=0}^n \binom{n}{i} W_i(K)\, \rho^i,
\end{equation}
known as the Steiner formula. The coefficients $W_i(K)$ are the
\emph{quermassintegrals} of $K$. In particular,
\[
W_0(K)=|K|, \qquad n W_1(K)=P(K), \qquad W_n(K)=\omega_n,
\]
where $P(K)$ denotes the perimeter (i.e., the $(n-1)$–Hausdorff measure of
$\partial K$), and $\omega_n = |B|$ is the volume of the unit ball.

When the boundary of $K$ is $C^2$, each $W_i(K)$ admits an integral
representation in terms of the principal curvatures of $\partial K$. Although this
will not be required explicitly, it provides geometric intuition on the nature of
the coefficients.

Quermassintegrals satisfy a parallel Steiner-type identity:
for every $j \in \{0,\dots,n-1\}$,
\begin{equation}
\label{Steiner_quermass}
W_j(K+\rho B) = \sum_{i=0}^{n-j} \binom{n-j}{i} W_{j+i}(K)\, \rho^i.
\end{equation}

In the special case $j=1$, the previous formula recovers the expansion of the
perimeter of the outer parallel set:
\begin{align*}
P(K+\rho B) &= n \sum_{i=0}^{n-1} \binom{n-1}{i} W_{i+1}(K)\, \rho^i \\[1.2ex]
&= P(K) + n(n-1) W_2(K)\, \rho + O(\rho^2).
\end{align*}
Consequently, the first variation of the perimeter is controlled by the second
quermassintegral:
\begin{equation}
\label{derivata_perimetro}
\lim_{\rho \to 0^+} \frac{P(K+\rho B) - P(K)}{\rho} = n(n-1) W_2(K).
\end{equation}

Another fundamental tool in shape optimization is provided by the
Aleksandrov–Fenchel inequalities, which for quermassintegrals read:
for all $0 \le i < j \le n-1$,
\begin{equation}
\label{Aleksandrov_Fenchel}
\left(\frac{W_j(K)}{\omega_n}\right)^{\frac{1}{\,n-j\,}}
\ge
\left(\frac{W_i(K)}{\omega_n}\right)^{\frac{1}{\,n-i\,}},
\end{equation}
with equality if and only if $K$ is a ball. When $i=1$ and $j=2$, this yields the
explicit estimate
\begin{equation}
\label{AF_W2}
W_2(K) \ge
n^{-\frac{n-2}{\,n-1\,}}\, \omega_n^{\frac{1}{\,n-1\,}}\,
P(K)^{\frac{n-2}{\,n-1\,}}.
\end{equation}

\subsection{Some useful Lemmas}

Let $\Omega \subset \mathbb{R}^n$ be a bounded, open, convex set.
For $t \in [0,r_\Omega]$, we denote the inner parallel set at distance $t$ by
\[
\Omega_t := \{x \in \Omega : d(x,\partial\Omega) > t\},
\]
where $d(x,\partial\Omega)$ is the Euclidean distance to the boundary, and
$r_\Omega$ is the \emph{inradius} of $\Omega$, i.e., the radius of the largest ball
contained in $\Omega$.

By the Brunn–Minkowski theorem and the concavity of the distance function on
convex domains, the mapping
\[
t \mapsto P(\Omega_t)^{\frac{1}{\,n-1\,}}
\]
is concave on $[0,r_\Omega]$, and hence absolutely continuous on
$(0,r_\Omega)$.

The next result provides a lower bound for the perimeter variation of inner
parallel sets.
\begin{lemma}
    \label{lemma_derivata_perimetro_1}
    Let $\Omega$ be a bounded, convex, open set in $\R^n$. Then for almost every $t \in (0,r_{\Omega})$
    \begin{equation}
    \label{eq_lemma_1}
        -\frac{d}{dt} P(\Omega_t) \geq n(n-1) W_2(\Omega_t),
    \end{equation}
    and equality holds if $\Omega$ is a ball.
    \end{lemma}
A convenient reformulation applies when a function depends only on the
distance.

Let $f:[0,+\infty)\to[0,+\infty)$ be a strictly increasing $C^1$ function
with $f(0)=0$, and set $u(x)=f(d(x,\partial\Omega))$. Then for any $t>0$
the superlevel set of $u$ can be identified with a parallel set of $\Omega$:
\[
E_t := \{x\in\Omega : u(x)>t\} = \Omega_{f^{-1}(t)}.
\]

\begin{lemma}
    \label{lemma_derivata_perimetro_2}
    Let $f \colon [0,+\infty) \to [0,+\infty)$ be a strictly increasing $C^1$ function with $f(0)=0$. Set $u(x)=f(d(x))$ and
    \[
    E_t = \Set{ x \in \Omega \, : \, u(x) > t } = \Omega_{f^{-1}(t)},
    \]
    then
    \begin{equation}
        \label{eq_lemma_2}
        -\frac{d}{dt} P(E_t) \geq (n-1) \frac{W_2(E_t)}{\abs{\nabla u}_{u=t}}.
    \end{equation}
\end{lemma}

As a consequence, $t\mapsto P(E_t)$ is absolutely continuous on every bounded
interval, being the composition of two absolutely continuous maps.

Finally, we observe that if you consider the ball with the same $(n-1)-$quermass, then we can compare all the others.

\begin{lemma}
\label{lemmadp}
    Let $\Omega$ a bounded, open and convex set of $\R^n$, and let $B$ be the ball for which $$W_{n-1}(\Omega)= W_{n-1}(B)$$. Then, for all $0 \leq j \leq n-1$, we have
    $$W_j(\Omega) \leq W_j(B).$$
\end{lemma}

\subsection{The radial cases}
In this subsection, we study a property of $\lambda^{RN}(A_{s,t})$ and $\lambda^{NR}(A_{s,t})$. For the sake of clarity, we recall their definitions:
\begin{equation*}
    \label{ARN}
    \lambda^{RN}(A_{s,t})=\min_{\substack{w\in H^{1}(A_{s,t}) \\ w \neq 0}} \frac{\displaystyle{\int_{A_{s,t}} \abs{\nabla w}^2\, dx+\beta_{in} \int_{\partial B_s} \abs{w}^2\, d\mathcal{H}^{n-1}}}{\displaystyle{\int_{A_{s,t}} \abs{w}^2\, dx}}.
\end{equation*}
and 
\begin{equation*}
    \label{ANR}
    \lambda^{NR}(A_{s,t})=\min_{\substack{w\in H^{1}(A_{s,t}) \\ w \neq 0}} \frac{\displaystyle{\int_{A_{s,t}} \abs{\nabla w}^2\, dx+\beta_{out} \int_{\partial B_t} \abs{w}^2\, d\mathcal{H}^{n-1}}}{\displaystyle{\int_{A_{s,t}} \abs{w}^2\, dx}}.
\end{equation*}
\begin{lemma}\label{incrdecr}
Let $0< R_1 \leq r \leq R_2$. Then, if $(\beta_{in}, \beta_{out}) \in (0, + \infty)^2$, it holds 
 $$r \mapsto \lambda^{RN}(A_{R_1,r}) \,\,\textnormal{is a decreasing function in}\, (R_1, R_2)$$

 $$r \mapsto \lambda^{NR}(A_{r,R_2}) \,\,\textnormal{is an increasing function in}\, (R_1, R_2);$$

\end{lemma}
\begin{proof}
Let $r_1<r_2$ and consider $v$ the eigenfunction associated to $\lambda^{RN}(A_{R_1,r_1})$. Let us extend $v$ in the spherical shell $A_{r_1,r_2}$, to obtain the function $\tilde v$ defined in $A_{R_1,r_2}$. Using $\tilde v$ as a test in the definition of $\lambda^{RN}(A_{R_1,r_2})$ we have \begin{equation*}
   \begin{aligned}
       \lambda^{RN}(A_{R_1,r_2})&\leq \frac{\displaystyle{\int_{A_{R_1,r_2}} \abs{\nabla \tilde v}^2\, dx+\beta_{in} \int_{\partial B_{R_1}} \abs{\tilde v}^2\, d\mathcal{H}^{n-1}}}{\displaystyle{\int_{A_{R_1,r_2}} \abs{\tilde v}^2\, dx}}\\ &\leq \frac{\displaystyle{\int_{A_{R_1,r_1}} \abs{\nabla v}^2\, dx+\beta_{in} \int_{\partial B_{R_1}} \abs{v}^2\, d\mathcal{H}^{n-1}}}{\displaystyle{\int_{A_{R_1,r_1}} \abs{v}^2\, dx}} = \lambda^{RN}(A_{R_1,r_1}).
   \end{aligned}
\end{equation*} 
The increasing nature of the function $r \mapsto \lambda^{NR}(A_{r,R_2})$ can be proved analogously.
\end{proof}

\begin{prop}\label{autofunzioni}
Let $0 < R_1 < R_2$ and $(\beta_{in}, \beta_{out}) \in (0, + \infty)^2$. The eigenfunctions corresponding to 
\[
\lambda^{RN} (A_{R_1,R_2}), \quad \lambda^{NR} (A_{R_1,R_2}), \quad \text{and} \quad \lambda^{RR} (A_{R_1,R_2})
\] 
have the following properties:
\begin{itemize}
    \item[(i)] the eigenfunction of $\lambda^{RN} (A_{R_1,R_2})$ is radially increasing;
    \item[(ii)] the eigenfunction of $\lambda^{NR} (A_{R_1,R_2})$ is radially decreasing;
    \item[(iii)] the eigenfunction of $\lambda^{RR} (A_{R_1,R_2})$ is radial and changes monotonicity exactly once.
\end{itemize}
\end{prop}
\begin{proof} We recall that a direct argument involving the characterization \eqref{ray} guarantees that, in all three cases, the first eigenvalue is simple and the corresponding eigenfunction may be chosen to be strictly positive.
Furthermore, due to the rotational invariance of the problem \eqref{probprinc}, the eigenfunctions corresponding to $\lambda^{RN}(A_{R_1,R_2}),\lambda^{NR}(A_{R_1,R_2}),\lambda^{RR}(A_{R_1,R_2})$ are radially symmetric. So, if $u>0$ is the first eigenfunction in $A_{R_1,R_2}$ then we can write $u(x)=:\psi(|x|)$.
\item[\it{(i)}] The radial problem can be written:
\begin{equation}
   \begin{cases}
       \displaystyle -\frac{1}{r^{n-1}}(\psi'(r)r^{n-1})'=\lambda^{RN}(A_{R_1,R_2})\psi(r) \,\,\,\, r\in(R_1,R_2)
       \\
      - \psi'(R_1)+\beta_{in}\psi(R_1)=0
       \\
       \psi'(R_2)=0.
   \end{cases} 
\end{equation}
Since $\psi>0$, then $(\psi'(r)r^{n-1})'<0$; from the boundary condition in $R_2$ it follows that $\psi'(r)>0$ for all $r\in(R_1,R_2)$
\item[\it{(ii)}] The argument to prove the descreasing nature of the eigenfunction $\lambda^{NR}(A_{R_1,R_2})$ is similar to that in $\it{(i)}$.
\item[\it{(iii)}] In this case we write the radial problem as
\begin{equation}
   \begin{cases}
       \displaystyle -\frac{1}{r^{n-1}}(\psi'(r)r^{n-1})'=\lambda^{RR}(A_{R_1,R_2})\psi(r) \,\,\,\, r\in(R_1,R_2)
       \\
      - \psi'(R_1)+\beta_{in}\psi(R_1)=0
       \\
       \psi'(R_2)+\beta_{out}\psi(R_2)=0.
   \end{cases} 
\end{equation}
Since $\psi>0$, then $\psi'(R_1)>0$ and $\psi'(R_2)<0$. Consequently, there exists $\rho\in(R_1,R_2)$ such that $\psi(\rho)=\displaystyle \max_{R_1<t<R_2}\psi(t)$. So $\psi'(\rho)=0$ and $\psi$ is a solution of the following problem
\begin{equation}
   \begin{cases}
       \displaystyle -\frac{1}{r^{n-1}}(\psi'(r)r^{n-1})'=\lambda^{RN}(A_{R_1,R_2})\psi(r) \,\,\,\, r\in(R_1,\rho)
       \\
      - \psi'(R_1)+\beta_{in}\psi(R_1)=0
       \\
       \psi'(\rho)=0.
   \end{cases} 
\end{equation}
From $\it{(i)}$ we have that in $(R_1,\rho)$ the function $\psi$ is only increasing. From an analogous result on $(\rho,R_2)$ the assertion follows.
\end{proof}

\subsection{Reverse Faber-Krahn for mixed eigenvalue involving a NBC}
In this subsection, we recall the results proved in \cite{DPP} and \cite{Paoli_Piscitelli_Trani}, 
and we discuss an extension needed for our purposes. In particular, in the spirit of \cite{bg2},\cite{dTAMS}, and \cite{L}, we can remove the 
Lipschitz regularity assumption on the sets involved, showing that the corresponding 
inequalities remain valid under weaker geometric assumptions.
Moreover, here we prove a the rigidity of such inequality.

More precisely, let $\beta_{\mathrm{in}} > 0$, and let $\Omega_{\mathrm{in}} \subset\subset \Omega_{\mathrm{out}} \subset \mathbb{R}^n$, where $\Omega_{\mathrm{in}}$ is Lipschitz and $\Omega_{\mathrm{out}}$ is an open set of finite perimeter. Let $\Omega = \Omega_{\mathrm{out}} \setminus \overline{\Omega}_{\mathrm{in}}$, we consider the Sobolev space $W^{1,2}(\Omega)$ as in \cite[Chapter~1]{mazya}, which consists of all functions $ u \in L_2(\Omega)$
such that its distributional derivative lies in $ u \in L_2(\Omega)$. Proceding as in \cite{dTAMS, BG}, we
define the first eigenvalue of $\Omega$ as follows
\begin{equation}
\label{RN}
\lambda^{RN}(\Omega)
= \inf_{\substack{\varphi \in W^{1,2}(\Omega) \\ \varphi \ne 0}}
\frac{\displaystyle 
      \int_\Omega |\nabla \varphi|^2\, dx 
      + \beta_{\mathrm{in}} \int_{\partial \Omega_{\mathrm{in}}} \varphi^2\, d\mathcal{H}^{n-1}
     }
     {\displaystyle 
      \int_\Omega \varphi^2\, dx
     }.
\end{equation}
 Following \cite{bg2}, this functional is well defined in $H^1(\Omega)$, and it holds the following Lemma.

\begin{lemma}\label{autofpos}
Let $\beta_{\mathrm{in}} > 0$, and let $\Omega_{\mathrm{in}} \subset\subset \Omega_{\mathrm{out}} \subset \mathbb{R}^n$, where $\Omega_{\mathrm{in}}$ is Lipschitz and $\Omega_{\mathrm{out}}$ is an open set of finite perimeter. Let $\Omega = \Omega_{\mathrm{out}} \setminus \overline{\Omega}_{\mathrm{in}}$, and let $\lambda^{RN}(\Omega)$ be defined in \eqref{RN}. 
    If there exists $\lambda \geq 0$ and  $w$ a positive function such that  
    
\begin{equation}
\label{weakrn}
    \int _\Omega \nabla w \nabla \varphi \, dx+\beta_{in} \int_{\partial\Omega_{in}} w \varphi \, d \mathcal{H}^{n-1} = \lambda\int_{\Omega} w \varphi \, dx, \qquad \forall \varphi \in W^{1,2}(\Omega),
    \end{equation}
then \begin{equation}
\label{lambda=rn}
    \lambda=\lambda^{RN}(\Omega)
\end{equation}
and 

\begin{equation*}
    \lambda^{RN}(\Omega)
=
\frac{\displaystyle 
      \int_\Omega |\nabla w|^2\, dx 
      + \beta_{\mathrm{in}} \int_{\partial \Omega_{\mathrm{in}}} w^2\, d\mathcal{H}^{n-1}
     }
     {\displaystyle 
      \int_\Omega w^2\, dx
     }.
\end{equation*}
\end{lemma}

\begin{proof}

 By \eqref{RN}, for all $\sigma>0$ there exists a function $v_\sigma$ such that $$\lambda^{RN}(\Omega)<\frac{\displaystyle \int_\Omega \abs{\nabla v_\sigma}^2 dx +\beta_{in}\int_{\partial \Omega_{in}}v_\sigma^2 \,\,d\mathcal{H}^{n-1}}{\displaystyle \int_{\Omega}v_\sigma^2 \,\,dx}<\lambda^{RN}(\Omega)+\sigma,$$
hence
\begin{equation}
\label{ineqvsigma}
  \displaystyle \int_\Omega \abs{\nabla v_\sigma}^2 dx +\beta_{in}\int_{\partial \Omega_{in}}v_\sigma^2 \,\,d\mathcal{H}^{n-1}<\left(\lambda^{RN}(\Omega)+\sigma\right)\int_{\Omega}v_\sigma^2 \,\,dx.
\end{equation}

Using the following function $\displaystyle z=\frac{v_\sigma^2}{w+\varepsilon}$, for $\varepsilon>0$, as a test in \eqref{weakrn}, we get 
\begin{equation}
\label{testz}
    \int_{\Omega} \frac{2 v_\sigma \nabla w \nabla v_\sigma}{w+\varepsilon} - \frac{\nabla w^2 v_\sigma^2}{(w+\varepsilon)^2}\,\,dx +\beta_{in} \int_{\partial \Omega_{in}}\frac{w}{w+\varepsilon}v_\sigma^2 \,\,d \mathcal{H}^{n-1}=\lambda\int_{\Omega}\frac{w}{w+\varepsilon}v_\sigma^2 \,\,dx.
\end{equation}

By subtracting \eqref{testz} by \eqref{ineqvsigma}, we obtain

\begin{equation}
\begin{aligned}
     0  &\leq \int_\Omega \abs{\nabla v_\sigma -\frac{ v_\sigma}{w+\varepsilon}\nabla w}^2dx= \int_\Omega \left(\abs{\nabla v_\sigma}^2 - \frac{2 v_\sigma \nabla w \nabla v_\sigma}{w+\varepsilon} + \frac{ v_\sigma^2}{(w+\varepsilon)^2}\abs{\nabla w}^2\right)\,\,dx   \\&<-\beta_{in} \int_{\partial \Omega_{in}}\frac{\varepsilon}{w+\varepsilon}v_\sigma^2 \,\,d \mathcal{H}^{n-1}+\left(\lambda^{RN}(\Omega)+\sigma\right)\int_{\Omega}v_\sigma^2 \,\,dx-\lambda\int_{\Omega}\frac{w}{w+\varepsilon}v_\sigma^2 \,\,dx.
    \end{aligned}
\end{equation}
Letting $\varepsilon \rightarrow 0$, it follows that
\begin{equation}
     \left(\lambda^{RN}(\Omega)+\sigma\right)\int_{\Omega}v_\sigma^2 \,\,dx-\lambda\int_{\Omega}v_\sigma^2 \,\,dx\geq 0,
\end{equation}
which implies $$\left(\lambda^{RN}(\Omega)+\sigma\right)-\lambda\geq 0.$$
In conclusion, taking $\sigma$ to zero, we get $$\lambda^{RN}(\Omega)\geq\lambda.$$
Using $w$ as a test in \eqref{RN}, we find the reverse inequality and so $$\lambda^{RN}(\Omega)=\lambda.$$

The last property follows by taking 
$w$ as a test function in \eqref{weakrn}.
\end{proof}

\begin{oss}
\label{tutticasi}
Let $\beta_{\mathrm{out}} > 0$, and let $\Omega_{\mathrm{in}} \subset\subset \Omega_{\mathrm{out}} \subset \mathbb{R}^n$, where $\Omega_{\mathrm{out}}$ is Lipschitz and $\Omega_{\mathrm{in}}$ is an open set of finite perimeter. Let $\Omega = \Omega_{\mathrm{out}} \setminus \overline{\Omega}_{\mathrm{in}}$, and consider the Sobolev space $W^{1,2}(\Omega)$ as in \cite[Chapter~1]{mazya}.

We define the first eigenvalue of $\Omega$ as
\begin{equation}
\label{Nr}
\lambda^{NR}(\Omega)
= \inf_{\substack{\varphi \in W^{1,2}(\Omega) \\ \varphi \ne 0}}
\frac{\displaystyle 
      \int_\Omega |\nabla \varphi|^2\, dx 
      + \beta_{\mathrm{out}} \int_{\partial \Omega_{\mathrm{out}}} \varphi^2\, d\mathcal{H}^{n-1}
     }
     {\displaystyle 
      \int_\Omega \varphi^2\, dx
     }.
\end{equation}

Lemma \ref{autofpos} still holds if one have $\lambda \geq 0$ and $w$ a positive function such that
\begin{equation}
\label{weaknr}
    \int _\Omega \nabla w \nabla \varphi \, dx+\beta_{out} \int_{\partial\Omega_{out}} w \varphi \, d \mathcal{H}^{n-1} = \lambda\int_{\Omega} w \varphi \, dx, \qquad \forall \varphi \in W^{1,2}(\Omega).
    \end{equation}

\end{oss}

For the sake of clarity, we now recall the result obtained by Della Pietra and Piscitelli~\cite[Theorem~1.1]{DPP}, extending it to a slightly larger class of sets.


\begin{teorema}[\cite{DPP}, Theorem~1.1] \label{DPP}
 Let $\Omega= \Omega_{out } \setminus \overline{\Omega}_{in}$ with $\Omega_{in}$ a bounded, open and convex set and $\Omega_{out}$ a set of finite perimeter, and let $\lambda^{RN}(\Omega)$ be the eigenvalue defined in \eqref{RN}.

    If  ${A}_{r,R}$  is spherical shell such that
 $$
\abs{\Omega}= \abs{{A}_{r,R}} \qquad 
    W_{n-1}( \Omega_{in}) =    W_{n-1}( B_r),$$
     then 
     \begin{equation}
         \label{piscitellide}
         \lambda^{RN} (\Omega) \leq \lambda^{RN}(A_{r,R}).
     \end{equation}

  Moreover, if equality occurs in \eqref{piscitellide} then, up to translations
     $$\Omega= A_{r,R}.$$
\end{teorema}
\begin{proof}
Let us denote $A:=A_{r,R}$ and let us consider the positive solution $v$ of the problem \eqref{RN} on $A$.
By the radial properties of $v$ described in Proposition \ref{autofunzioni}, the function
    \[
    g(t)= \abs{\nabla v}_{v=t}
    \]
    is well defined for all $t\in (v_m, v_M)$, where $\displaystyle{v_m:=\min_{A} v}$ and $\displaystyle{v_M:= \max_{A} v}$.
    Let us define 
    \begin{equation*}
        \label{fug}
        G^{-1}(t)=\int_{v_m}^{t} \frac{1}{g(s)}\, ds. \qquad v_m < t <v_M.
    \end{equation*}
    Since $G^{-1}$ is strictly increasing, its inverse $G$ is well defined and it is an increasing function satisfying $G(0)=v_m$. So we can construct, for $x \in \Omega$,\begin{equation}
        \label{w=gd}
        w(x) =\begin{cases}
             G( d(x, \Omega_{in})) \, & \text{if }  d(x, \Omega_{in}) \leq R-r;\\
             v_M & \text{if } d(x, \Omega_{in}) > R-r.
        \end{cases}
    \end{equation}
    By construction, $u\in W^{1,p}(\Omega)$ and

    \begin{itemize} 
        \item $w_m:=\displaystyle \min_{\Omega} w= G(0)=v_m=\min_{A} v$;

        \item $w_M:=\displaystyle \max_{\Omega} w<v_M;$ 
        
        \item the chain rule formula implies 
        \[
        \abs{\nabla w}_{w=t}= \abs{G'(d(x))}_{w=t} = \Bigl \lvert g \Bigl( G \bigl (d(x ) \bigr) \Bigr) \Bigr \rvert_{w=t} = g(t)= \abs{\nabla v}_{v=t}.
        \]
        
    \end{itemize}

If we denote by $$E_t \subset \{x \in \Omega\, : \, w <t\}, \qquad F_t=\overline B_{r} \cup \{x \in A_{r,R}\, : \, v <t\},$$
it is not difficult to observe that

$$E_t\subset  \{x \in \R^{n}\, : \, d(x) <G^{-1}(t)\}:= \tilde E_t, \qquad F_t=\{x \in \R^n\, : \, \abs{x} <r+G^{-1}(t)\}.$$

By the Steiner formula  \eqref{Steiner_formula} and the Alexandrov-Fenchel inequalities \eqref{Aleksandrov_Fenchel}, we get, as $\rho = G^{-1}(t)$, that

\begin{equation}
\label{p<p}
    \begin{multlined}
    \mathcal{H}^{n-1}(\partial E_t \cap \Omega) \leq P(\tilde E_t) = P(D+\rho B_1) = n \sum_{k=0}^{n-1} \binom{n-1}{k} W_{k+1}(\Omega_{in})\rho^k \\
        \leq n \sum_{k=0}^{n-1} \binom{n-1}{k} W_{k+1}(B_r)\rho^k  = P(B_r +\rho B_1)= P(F_t).
    \end{multlined}
\end{equation}
 By the coarea formula, \eqref{p<p} and the fact that $u(x)=u_m=v_m$ on $\partial \Omega_{in}$, we have
 \begin{equation}
     \label{disnormabordo}
      \begin{gathered}
     \int_{\Omega} \abs{\nabla w}^2 \, dx = \int_{v_m}^{v_M} g(t)\mathcal{H}^{n-1}(\partial E_t \cap \Omega) \, dt \leq \int_{v_m}^{v_M} g(t) P(F_t) \, dt = \int_{A_{r,R}} \abs{\nabla v}^2 \, dx;\\
     \int_{\partial \Omega_{in}} w^2 \, d\mathcal{H}^{n-1} = w^2P(\Omega_{in}) \leq v^2P(A_{r,R}) =\int_{\partial B_r} v^2 \, d\mathcal{H}^{n-1}.
 \end{gathered}
 \end{equation}

Now, we define $\mu(t)= \abs{E_t \cap \Omega}$ and $\nu(t) = \abs{F_t \cap A_{r,R}}$, and using again the coarea formula for $v_m \leq t < v_M$, we get
\begin{equation}
    \begin{multlined}
        \mu'(t) = \int_{\{w=t\} \cap \Omega} \frac{1}{\abs{\nabla w}} \, d \mathcal{H}^{n-1} = \frac{ \mathcal{H}^{n-1}( \partial E_t) \cap \Omega}{ g(t)} = \frac{ P( \tilde E_t) }{ g(t)}
        \leq \frac{ P( F_t) }{ g(t)} = \int_{\{v=t\}} \frac{1}{\abs{\nabla v}} \, d \mathcal{H}^{n-1} = \nu'(t).
    \end{multlined}
\end{equation}

Since for $t<v_m$ we have $\mu(t)=\nu(t)=0$, by integrating we have for $0 < t <v_M$
$$\mu(t) \leq \nu(t).$$
 Hence, we get
    \begin{equation}
    \label{disnorma2}
    \displaystyle \int_\Omega w^2 dx= \int_0^{v_M} 2t(|\Omega|-\mu(t))dt\ge\int_0^{v_M} 2t(|A|-\eta(t))dt=\int_\Omega v^2dx.
    \end{equation}
    Using \eqref{disnormabordo} and \eqref{disnorma2}, we obtain
    \begin{equation}
    \label{a24}
    \lambda^{RN}(\Omega)\le \frac{\displaystyle \int_\Omega \abs{\nabla w}^2 dx +\beta_{in}\int_{\partial \Omega_{in}}w^2 \,\,d\mathcal{H}^{n-1}}{\displaystyle \int_{\Omega}w^2 \,\,dx}\le \frac{\displaystyle \int_A \abs{\nabla v}^2 dx +\beta_{in}\int_{\partial B_r}v^2 \,\,d\mathcal{H}^{n-1}}{\displaystyle \int_{A}v^2 \,\,dx}= \lambda^{RN}(A).
    \end{equation}
    To conclude the proof, we analyze the equality case in \eqref{piscitellide} and establish rigidity. 
First, we observe that this implies equality in \eqref{p<p}, as well as in the Aleksandrov--Fenchel inequality for $\Omega_{in}$. Consequently, up to translations, we obtain
\[
\Omega_{in} = B_{r}.
\]
Moreover, using $\lvert \Omega \rvert = \lvert A_{r,R} \rvert$, we deduce that $\Omega_{out}$ is also a ball, not necessarily concentric with $B_r$ and it holds
\begin{equation}
\label{weakrn2}
\int_{\Omega} \nabla w \cdot \nabla \varphi \, dx 
+ \beta_{in} \int_{\partial B_r} w \varphi \, d\mathcal{H}^{n-1}
= \lambda^{RN}(\Omega) \int_{\Omega} w \varphi \, dx,
\qquad \forall \varphi \in W^{1,2}(\Omega).
\end{equation}

We now have to show that $\Omega_{out} = B_R$, i.e.\ it is concentric with $B_r$.

We argue by contradiction. Assume that $\Omega_{out} \neq B_R$. If we consider  a nonnegative test function $\varphi$, not identically zero and compactly supported in the interior of $\Omega_{out} \setminus \overline{B}_R$, since $w$ is a positive constant in $\Omega_{out} \setminus \overline{B}_R$, we obtain
\begin{equation}
0 = \int_{\Omega} \nabla w \cdot \nabla \varphi \, dx 
+ \beta_{in} \int_{\partial B_r} w \varphi \, d\mathcal{H}^{n-1}
= \lambda^{RN}(\Omega) \int_{\Omega} w \varphi \, dx > 0,
\end{equation}
which is a contradiction.
Therefore, $\Omega_{out} = B_R$, and it is concentric with $\Omega_{in} = B_r$.

\end{proof}

Finally, we recall the results concerning the Neumann–Robin setting under weaker assumptions then those considered in \cite[Theorem~3.1]{Paoli_Piscitelli_Trani}. Also in this case, the proof differs only slightly from the original argument.

\begin{teorema}\label{PPT}
    Let $\Omega_{in} \subset \subset \Omega_{out} \subset \R ^n$ two open sets, with $\Omega_{in}$ having finite perimeter and $\Omega_{out}$ being convex. Let $\Omega = \Omega_{out} \setminus\overline{\Omega}_{in}$ and let $\lambda^{NR}(\Omega)$ denote the first mixed Neumann-Robin eigenvalue, with $\beta_{out}\in \mathbb{R}\setminus \{0\}$.
    If $R>r>0$ are such that the spherical shall ${A}_{r,R}$ has the properties
    $$
    \abs{\Omega}= \abs{{A}_{r,R}}, \qquad \mathcal{H}^{n-1}(\partial B_R)=\mathcal{H}^{n-1}(\partial \Omega_{out}), $$ 
then we have
     \begin{equation}
         \label{ptteq}
     \lambda^{NR}(\Omega)\leq \lambda^{NR}({A}_{r,R}).
     \end{equation}

      Moreover, if equality occurs in \eqref{ptteq} then, up to translations
     $$\Omega= A_{r,R}.$$
\end{teorema}

\section*{Acknowledgments}
The authors would like to thank Cristina Trombetti for helpful discussions on this problem.

The authors were partially supported by Gruppo Nazionale per l’Analisi Matematica, la Probabilità e le loro Applicazioni
(GNAMPA) of Istituto Nazionale di Alta Matematica (INdAM).

\bibliographystyle{plain}
\bibliography{biblio}
\Addresses
\end{document}